\documentclass[11pt]{article}
\usepackage[top=2.54cm, bottom=2.54cm, left=2.70cm, right=2.70cm]{geometry}
\usepackage[tbtags]{amsmath}
\usepackage{amssymb}
\usepackage{amsthm}
\usepackage{mathrsfs}
\usepackage{fancyhdr}
\usepackage{float}
\usepackage{graphicx}
\usepackage{tikz}
\usepackage{tkz-graph}
\usepackage{tkz-berge}
\usetikzlibrary{decorations.pathreplacing}
\usepackage[numbers,sort&compress]{natbib}
\usepackage{setspace}

\allowdisplaybreaks[4]
\newtheorem{theorem}{\bf Theorem}[section]

\newtheorem{problem}[theorem]{Problem}

\newtheorem{conj}[theorem]{Conjecture}

\newtheorem{claim}{\indent Claim}[]

\newcommand{\smalldoublearrow}[1]{\overset{\text{\scalebox{1.0}{$\leftrightarrow$}}}{#1}}

\newcommand{\smallarrow}[1]{\overset{\text{\scalebox{1.0}{$\rightarrow$}}}{#1}}

\DeclareMathOperator{\init}{init}
\DeclareMathOperator{\term}{term}
\DeclareMathOperator{\dist}{dist}

\begin{document}
	
	\begin{spacing}{1.1}
		
		\title{An improvement bound on a problem of Picasarri-Arrieta and Rambaud}
		
		\author{Bin Chen$^a$, \, Xinmin Hou$^{b,c}$\thanks {Email address: xmhou@ustc.edu.cn(X. Hou)}, \, Yue Ma$^d$, \,Zhi Yin$^b$, \,Xinyu Zhou$^b$\\
			\small$^a$ School of Mathematics and Statistics, Fuzhou University, Fujian, China\\
			\small$^b$ School of Mathematical Sciences,\\
			\small University of Science and Technology of China, Hefei 230026, Anhui, China\\
			\small$^c$ Hefei National Laboratory,\\
\small University of Science and Technology of China, Hefei 230088, Anhui, China\\
			\small$^d$ School of Mathematics and Statistics,\\
			\small Nanjing University of Science and Technology, Nanjing 210094, Jiangsu, China\\}
		\date{}
		
		\maketitle
		
		\noindent{\bf Abstract}\,: Let $k$ and $\ell$ be positive integers. A cycle with two blocks $C(k,\ell)$
		is a digraph consisting of two internally vertex disjoint directed paths of lengths $k$ and $\ell$ with the same initial vertex and terminal vertex. Picasarri-Arrieta and Rambaud (European J. Combin., 2024) proved that for any $k\geq 2$, every digraph of minimum out-degree at least two and girth at least $8k-6$ contains a subdivision of $C(k,k)$. They also construct a family of digraphs showing that the girth cannot be reduced to $k-1$, and posed the problem of determining  the minimum girth such that every digraph of minimum out-degree at least two contains a subdivision of $C(k,k)$. In this paper, we improve the lower bound on the girth from  $8k-6$ to $4k+2$,  
        and construct a family of digraphs in which every member has minimum out-degree two and girth $k$ but contains no subdivision of $C(k,k)$. Thus our results show that the  girth in question lies between $k+1$ and $4k+2$.

		{\bf AMS}\,: 05C20; 05C38.
		
		
		\section{Introduction}
		
		\noindent

		Given a graph $F$, a subdivision of $F$ is a graph obtained from $F$ by replacing its edges by internally vertex disjoint paths. Let $K_{k}$ be the complete graph on $k$ vertices. Mader \cite{Mader67} 
        proved that for any positive integer $k$ there exists an integer $g(k)$ such that every graph of average degree at least $g(k)$ contains a subdivision of $K_{k}$. The rate of the growth of $g(k)$ was independently determined by Bollob\'{a}s and Thomason \cite{BoT} and by Koml\'{o}s and Szemer\'{e}di \cite{KS}. Specifically, they proved that there exists a constant $C>0$ such that every graph with average degree $Ck^{2}$ contais a subdivision of $K_{k}$, and that this bound is best possible up to the constant factor $C$.

		There is a natural analogue of subdivisions in digraphs. For a digraph $F$, a subdivision of $F$ is a digraph obtained by replacing every arc of $F$ by a directed path from its tail to head such that all subdivision directed paths are pairwise internally vertex disjoint. Researchers are interested in finding subdivisions of specific digraphs with respect to several digraph parameters. It is convenient to use the following notation from \cite{Aboulker}. Given a digraph parameter $\gamma$, a digraph $F$ is said to be \emph{$\gamma$-maderian} if there exists a least integer mader$_{\gamma}(F)$ such that
		every digraph $D$ with $\gamma(D)$$\geq$ mader$_{\gamma}(F)$ contains a subdivision of $F$ as a subdigraph. We call mader$_{\gamma}(F)$ the \emph{Mader number} of $F$ with respect to the digraph parameter $\gamma$.
		
		In digraph theory, one of the most fundamental and widely studied parameters is the minimum out-degree. Let us introduce some notation before presenting the classical results. An \emph{orientation} of a graph $G$, denoted by $\smallarrow{G}$, is an assignment of exactly one direction to each of its edges. The \emph{biorientation} of a graph $G$, denoted by $\smalldoublearrow{G}$, is the digraph formed by replacing each of its edges by two arcs with opposite directions. By $k\smalldoublearrow{K}_{2}$ we mean a digraph consisting of $k$ vertex disjoint $\smalldoublearrow{K}_{2}$. A celebrated conjecture of Bermond and Thomassen \cite{BT} is that for any positive integer $k$, mader$_{\delta^+}(k\smalldoublearrow{K}_{2})=2k-1$. This conjecture is trivially true when $k=1$, and it has been confirmed to be true for the cases $k=2$ by Thomassen \cite{TH} and $k=3$ by Lichiardopol, P\'{o}r and Sereni \cite{LPS}. For general $k\geq 4$, there is no exact result. Thomassen \cite{TH} initially proved that $k\smalldoublearrow{K}_{2}$ is $\delta^{+}$-maderian. Specifically, he showed that mader$_{\delta^+}(k\smalldoublearrow{K}_{2})\leq (k+1)!$. Subsequently, Alon \cite{Alon} greatly reduced this estimate to mader$_{\delta^+}(k\smalldoublearrow{K}_{2})\leq 64k$ by combining probabilistic method and structural analysis. The above bound was further improved to mader$_{\delta^+}(k\smalldoublearrow{K}_{2})\leq 18k$ by Buci\'{c} \cite{B}. It was proved by Gishboliner, Steiner and Szab\'{o} \cite{Gishboliner22_2} that any digraph with minimum out-degree greater than one contains a subdivision of $\smalldoublearrow{K}_{3}$ minus an arc. Let $S_{n}$ be a star on $n$ vertices with a central vertex and $n-1$ leaves. Thomassen in \cite{Thomassen} determined that mader$_{\delta^{+}}(\smalldoublearrow{S_{3}})=2$.

		Whereas, there are also many digraphs that are not $\delta^{+}$-maderian. Thomassen \cite{Th85} constructed a family of digraphs with  arbitrarily large minimum out-degree  but no directed cycle of even length. This implies that  even digraphs (those for which every subdivisions contains a directed cycle of even length, with $\smalldoublearrow{K}_{3}$ being the smallest example) are not $\delta^{+}$-maderian. He also constructed another family of digraphs with arbitrarily high minimum out-degree in which every digraph has no subdivision of $\smalldoublearrow{S_{4}}$, that is, $\smalldoublearrow{S_{k}}$ is not $\delta^{+}$-maderian for any $k\geq 4$.

		A digraph is \emph{acyclic} if there exists no directed cycle. Surprisingly, no acyclic digraph is known that fails the property of $\delta^{+}$-maderian. In \cite{Mader85}, Mader put forward the following famous conjecture.

		\begin{conj}\label{main conjecture}
			
			Every acyclic digraph is $\delta^{+}$-maderian.
			
		\end{conj}

		A \emph{tournament} is a digraph in which there is exactly one arc between every pair of distinct vertices. A \emph{transitive tournament} is a tournament without directed cycles. Let $TT_{n}$ be a transitive tournament on $n$ vertices. To verify Conjecture~\ref{main conjecture}, it suffices to show that $TT_{n}$ is $\delta^{+}$-maderian for any given $n$. Evidently, mader$_{\delta^{+}}(TT_{n})=1$ for $n=1,2$ and mader$_{\delta^{+}}(TT_{3})=2$. In \cite{Ma96}, Mader further determined that mader$_{\delta^{+}}(TT_{4})=3$. However, there is no result about $TT_{n}$ when $n\geq 5$, we even do not know whether mader$_{\delta^{+}}(TT_{n})$ exists or not.

		We next present some progress regarding Conjecture~\ref{main conjecture}. An \emph{out-arborescence}\;(resp., \emph{in-arborescence}) is an oriented tree with a vertex, called its \emph{root}, such that all arcs are directed away from (resp., towards) the root. Observe that every out-arborescences is $\delta^{+}$-maderian. In \cite{Aboulker}, Aboulker, Cohen, Havet, Lochet, Moura and Thomass\'{e} proved that any in-arborescences is also $\delta^{+}$-maderian. In the same paper, they further deduced that every oriented path is $\delta^{+}$-maderian. Indeed, they determined that mader$_{\delta^{+}}(F)=n-1$ when $F$ is an oriented path on $n$ vertices.

		In 1995, Mader \cite{Ma95} proved that for all $t\in \mathbb{N}$, the digraphs
		consisting of $t$ internally disjoint directed paths with the same initial vertex and the terminal vertex are $\delta^{+}$-maderian. But we do not know how long these directed paths can be. It is natural and interesting to investigate this problem, and the first nontrivial case is $t=2$ due to Aboulker et al. \cite{Aboulker}.
		Let $C(k,\ell)$ be an oriented cycle formed by two internally vertex disjoint directed paths with the same initial vertex and terminal vertex of lengths $k$ and $\ell$, respectively, where $k$ and $\ell$ are positive integers. 
		 Note that the complete digraph $\smalldoublearrow{K}_{k+\ell-1}$ has minimum out-degree $k+\ell-2$ and contains no subdivision of $C(k,\ell)$, as $|V(C(k,\ell))|=k+\ell$. Consequently, mader$_{\delta^{+}}(C(k,\ell))\geq k+\ell-1$. Also in \cite{Aboulker}, Aboulker et al. showed that mader$_{\delta^{+}}(C(k,\ell))\leq 2k+2\ell-1$, and they further raised a problem to determine mader$_{\delta^{+}}(C(k,\ell))$.
		
		\begin{problem}[\cite{Aboulker}]\label{kl}
			For positive integers $k$ and $\ell$, what is the value of mader$_{\delta^{+}}(C(k,\ell))$?
		\end{problem}
		
		In 2022, Gishboliner, Steiner and Szab\'{o} \cite{Gishboliner22_2} improved the upper bound by showing that mader$_{\delta^{+}}(C(k,\ell))\leq k+3\ell-1$, where $k\geq \ell$. Subsequently, Liu and Yoo \cite{Liu} determined that mader$_{\delta^{+}}(C(k,\ell))=k+\ell-1$ for any positive integers $k$ and $\ell$. More generally, Gishboliner et al. \cite{Gishboliner22_2} proved that any oriented cycle is $\delta^{+}$-maderian.
		
		The \emph{girth} of $D$, denoted by $g(D)$, is the length of a shortest directed cycle in $D$. In \cite{Picasarri-Arrieta}, Picasarri-Arrieta and Rambaud \cite{Picasarri-Arrieta} considered the subdivisions of $C(k,k)$ under the condition of large girth. For any integer $g$, we denote mader$_{\delta^+}^{(g)}(F)$ to be the least integer $k$ such that every digraph $D$ satisfying $\delta^+(D) \geq k$ and $g(D) \geq g$ contains a subdivision of $F$. Picasarri-Arrieta et al. presented the following result.
		
		\begin{theorem}[\cite{Picasarri-Arrieta}]\label{largegirth}
			For any positive integer $k$, mader$_{\delta^+}^{(8k-6)}(C(k,k))=2$.
		\end{theorem}
		
		On the other hand, Picasarri-Arrieta et al. also constructed a family of digraphs in which every digraph has minimum out-degree two and girth $k-1$ but contains no subdivision of $C(k,k)$, which implies that mader$_{\delta^+}^{(k-1)}(C(k,k))\geq 3$, and they proposed a problem as below.
		
		\begin{problem}[\cite{Picasarri-Arrieta}]\label{girth}
			Find the minimum value $g\in [k,8k-6]$ such that mader$_{\delta^+}^{(g)}(C(k,k))\leq 2$. 
		\end{problem}

		In this paper, we strengthen  Theorem~\ref{largegirth} by reducing the girth  condition from $8k-6$ to $4k+2$.
		\begin{theorem}\label{main theorem}
			For any positive integer $k$, mader$_{\delta^+}^{(4k+2)}(C(k,k))=2$.
		\end{theorem}

		Moreover, we show that the girth condition can not be reduced to $k$ by constructing a family of digraphs with $g(D)\ge k$ and $\delta^+(D)=2$ that contain no $C(k,k)$. 
		\begin{theorem}\label{main theorem2}
			For any positive integer $k$, mader$_{\delta^+}^{(k)}(C(k,k))\geq 3$.
		\end{theorem}
		\begin{proof} It is sufficient to give a family of digraphs of minimum out-degree at least two and girth at least $k$ such that every member contains no subdivision of $C(k,k)$.\\ 
			\noindent{\bf Construction $\mathcal{D}_k$:} Let $v_0,v_1,\ldots,v_{k-1}$ be $k$ vertices and $C_0,C_1,\ldots,C_{k-1}$ be $k$ cycles of the same length $k$. Identify $i$ with $i\bmod k$ for every integer $i$ and we construct the desired digraphs as below\;(see Figure~\ref{Fig1}): For each $i\in [0,k-1]$, select two arbitrary vertices of $C_i$ to be out-neighbors of $v_i$. Moreover, every vertex of $C_i$ is the in-neighbor of $v_{i+1}$. 
			
			One can check that every $D\in\mathcal{D}_k$ satisfies that $\delta^+(D)=2$ and $g(D)=k$ without subdivision of $C(k,k)$, as desired.
		\end{proof}
		
		\begin{figure}[hbtp]
			\begin{center}	
				\begin{tikzpicture}[thick,scale=0.618, every node/.style={transform shape}]\label{fig_3}			
					\tikzset{vertex/.style = {circle,fill=black,minimum size=4pt, inner sep=0pt}}
					\tikzset{edge/.style = {->,> = latex'}}
					\node[vertex, label=180:$v_0$] (v0) at (180:5) {};
					\begin{scope}[shift={(135:4)}]
						\node[vertex, label={[label distance=0.4em]90:$C_0$}] (C01) at (90:2) {};
						\node[vertex] (C02) at (135:1.2649110640673517327995574177731) {};
						\coordinate (C03) at (180:1) {};
						\coordinate (C04) at (225:1.2649110640673517327995574177731) {};
						\coordinate (C05) at (270:2) {};
						\coordinate (C06) at (315:1.2649110640673517327995574177731) {};
						\coordinate (C07) at (360:1) {};
						\coordinate (C08) at (45:1.2649110640673517327995574177731) {};
						\draw[edge] (C01) to[out=180, in=80.88975466] (C02);
						\draw[thick,edge,dashed] (C02) to[out=-99.11024534, in=90] (C03) to[out=-90, in=99.11024534] (C04) to[out=-80.88975466, in=-180] (C05) to[out=0, in=-99.11024534] (C06) to[out=80.88975466, in=-90] (C07) to[out=90, in=-80.88975466] (C08) to[out=99.11024534, in=0] (C01);
					\end{scope}
					\draw[edge] (v0) to (C04);
					\draw[edge] (v0) to (C05);
					\node[vertex, label=90:$v_1$] (v1) at (90:5) {};
					\draw[edge] (C01) to (v1);
					\draw[edge] (C02) to (v1);
					\draw[edge] (C03) to (v1);
					\draw[edge] (C04) to (v1);
					\draw[edge] (C05) to (v1);
					\draw[edge] (C08) to (v1);
					\begin{scope}[shift={(45:4)}]
						\node[vertex, label={[label distance=0.4em]90:$C_1$}] (C11) at (90:2) {};
						\node[vertex] (C12) at (135:1.2649110640673517327995574177731) {};
						\coordinate (C13) at (180:1) {};
						\coordinate (C14) at (225:1.2649110640673517327995574177731) {};
						\coordinate (C15) at (270:2) {};
						\coordinate (C16) at (315:1.2649110640673517327995574177731) {};
						\coordinate (C17) at (360:1) {};
						\coordinate (C18) at (45:1.2649110640673517327995574177731) {};
						\draw[edge] (C11) to[out=180, in=80.88975466] (C12);
						\draw[thick,edge,dashed] (C12) to[out=-99.11024534, in=90] (C13) to[out=-90, in=99.11024534] (C14) to[out=-80.88975466, in=-180] (C15) to[out=0, in=-99.11024534] (C16) to[out=80.88975466, in=-90] (C17) to[out=90, in=-80.88975466] (C18) to[out=99.11024534, in=0] (C11);
					\end{scope}
					\draw[edge] (v1) to (C12);
					\draw[edge] (v1) to (C14);
					\node[vertex, label=0:$v_2$] (v2) at (0:5) {};
					\draw[edge] (C11) to (v2);
					\draw[edge] (C12) to (v2);
					\draw[edge] (C13) to (v2);
					\draw[edge] (C14) to (v2);
					\draw[edge] (C15) to (v2);
					\draw[edge] (C16) to (v2);
					\draw[edge] (C18) to (v2);
					\node at (-10:4.37) [scale=1.5][rotate=45]{$\dots\dots$};
					\begin{scope}[shift={(-45:4)}]
						\node[vertex, label={[label distance=0.4em]90:$C_{k-2}$}] (Ck-1) at (90:2) {};
						\node[vertex] (Ck-2) at (135:1.2649110640673517327995574177731) {};
						\coordinate (Ck-3) at (180:1) {};
						\coordinate (Ck-4) at (225:1.2649110640673517327995574177731) {};
						\coordinate (Ck-5) at (270:2) {};
						\coordinate (Ck-6) at (315:1.2649110640673517327995574177731) {};
						\coordinate (Ck-7) at (360:1) {};
						\coordinate (Ck-8) at (45:1.2649110640673517327995574177731) {};
						\draw[edge] (Ck-1) to[out=180, in=80.88975466] (Ck-2);
						\draw[thick,edge,dashed] (Ck-2) to[out=-99.11024534, in=90] (Ck-3) to[out=-90, in=99.11024534] (Ck-4) to[out=-80.88975466, in=-180] (Ck-5) to[out=0, in=-99.11024534] (Ck-6) to[out=80.88975466, in=-90] (Ck-7) to[out=90, in=-80.88975466] (Ck-8) to[out=99.11024534, in=0] (Ck-1);
					\end{scope}
					\node[vertex, label=below:$v_{k-1}$] (vk-1) at (-90:5) {};
					\draw[edge] (Ck-1) to (vk-1);
					\draw[edge] (Ck-2) to (vk-1);
					\draw[edge] (Ck-4) to (vk-1);
					\draw[edge] (Ck-5) to (vk-1);
					\draw[edge] (Ck-6) to (vk-1);
					\draw[edge] (Ck-7) to (vk-1);
					\draw[edge] (Ck-8) to (vk-1);
					\begin{scope}[shift={(-135:4)}]
						\node[vertex, label={[label distance=0.4em]90:$C_{k-1}$}] (Ck1) at (90:2) {};
						\node[vertex] (Ck2) at (135:1.2649110640673517327995574177731) {};
						\coordinate (Ck3) at (180:1) {};
						\coordinate (Ck4) at (225:1.2649110640673517327995574177731) {};
						\coordinate (Ck5) at (270:2) {};
						\coordinate (Ck6) at (315:1.2649110640673517327995574177731) {};
						\coordinate (Ck7) at (360:1) {};
						\coordinate (Ck8) at (45:1.2649110640673517327995574177731) {};
						\draw[edge] (Ck1) to[out=180, in=80.88975466] (Ck2);
						\draw[thick,edge,dashed] (Ck2) to[out=-99.11024534, in=90] (Ck3) to[out=-90, in=99.11024534] (Ck4) to[out=-80.88975466, in=-180] (Ck5) to[out=0, in=-99.11024534] (Ck6) to[out=80.88975466, in=-90] (Ck7) to[out=90, in=-80.88975466] (Ck8) to[out=99.11024534, in=0] (Ck1);
					\end{scope}
					\draw[edge] (vk-1) to (Ck6);
					\draw[edge] (vk-1) to (Ck8);
					\draw[edge] (Ck1) to (v0);
					\draw[edge] (Ck2) to (v0);
					\draw[edge] (Ck4) to (v0);
					\draw[edge] (Ck5) to (v0);
					\draw[edge] (Ck6) to (v0);
					\draw[edge] (Ck7) to (v0);
					\draw[edge] (Ck8) to (v0);
				\end{tikzpicture}
				\caption{$\mathcal{D}_k$ (The solid arcs represent the arcs and the dashed arcs represent the paths.)}\label{Fig1}
			\end{center}
		\end{figure}
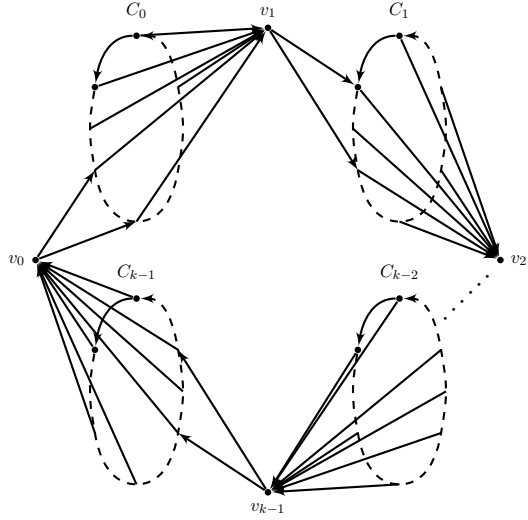

		The rest of the paper is organized as follows. In next section, we mainly give some notation and definitions. The proof of Theorem~\ref{main theorem} will be presented in Section 3. In the last section, we will give some remarks and discussions.

		\section{Preliminaries}
		
		\noindent
		
		Let $D=(V,A)$ be a digraph. The order of $D$ is denoted by $|D|=|V(D)|$ and its number of arcs is denoted by $\|D\|=|A(D)|$. We write $v\in D$ for $v\in V(D)$ and $uv\in D$ for $(u,v)\in A(D)$. 
		For every vertex $v\in D$, the \emph{out-neighborhood} of $v$ is the set $N_{D}^{+}(v)=\{u\in D:vu\in D\}$. Similarly, the \emph{in-neighborhood} of $v$ is the set $N_{D}^{-}(v)=\{u\in D : uv\in D\}$. The vertices in $N^{+}_{D}(v)$ and $N^{-}_{D}(v)$ are called the \emph{out-neighbors} and \emph{in-neighbors} of $v$, respectively. Define $d_{D}^+(v)=|N^{+}_{D}(v)|$ and $d_{D}^-(v)=|N^{-}_{D}(v)|$ as the out-degree and in-degree of $v$, respectively. We also define $\delta^+(D) = \min\limits_{v \in D} d^+_{D}(v)$, $\delta^-(D) = \min\limits_{v \in D} d^-_{D}(v)$, $\Delta^+(D) = \max\limits_{v \in D} d^+_{D}(v)$ and $\Delta^-(D) = \max\limits_{v \in D} d^-_{D}(v)$. For integers $a$ and $b$, we use the notation $[a,b]=\{a, a+1, \dots, b\}$, and for convenience, we further denote $[a]=[1,a]$.
		
		An \emph{oriented graph} is a digraph such that there is at most one arc between every pair of distinct vertices. An \emph{orientation} of a graph $G$ is an assignment of exactly one direction to each edge of $G$. A \emph{directed path}\;(\emph{path} for short) is an oriented path such that any two consecutive arcs have the same orientation. Analogously, a \emph{directed cycle}\;(\emph{cycle} for short) is an oriented cycle such that every pair of two consecutive arcs have the same orientation. A digraph is \emph{acyclic} if it has no cycle, otherwise we say it is nonacyclic. If $D$ is a nonacyclic digraph, then we denote $g(D)$ to be the \emph{girth} of $D$, which is the length of the shortest cycles of $D$. We further define $g(D)=+\infty$ if $D$ is acyclic.
		
		Given a path $P$, we denote by $\init(P)$ the initial vertex of $P$ and $\term(P)$ the terminal vertex of $P$. Let $Q$ be another path with $V(P)\cap V(Q)=\{\term(P)\} = \{\init(Q)\}$. By $P\circ Q$ we mean the concatenation of $P$ and $Q$, namely, the path obtained by first passing through every vertex of $P$ and then passing through those of $Q$.
		A $(u,v)$-path in $D$ is a path with initial vertex $u$ and terminal vertex $v$. Similarly, for two vertex sets $A,B\subseteq V(D)$, an $(A,B)$-path in $D$ is a path with initial vertex in $A$ and terminal vertex in $B$.
		The \emph{distance} from $u$ to $v$ in $D$, denoted by $\dist_D(u,v)$, is the length of a shortest $(u,v)$-path in $D$. We further denote $\dist_D(u,v)=+\infty$ if no such path exists.
		A {\it $(u,v)$-cut-vertex} $c$ is a vertex such that $\dist_{D-c}(u,v)=+\infty$.
		
		Given two vertices $a,b$ in a cycle $C$ (with possibly $a=b$), we denote by $C[a,b]$ the path from $a$ to $b$ along $C$ (which is the single-vertex path when $a=b$).
		Moreover, we define $C[a,b) = C[a,b]-\{b\}$, $C(a,b] = C[a,b]-\{a\}$ and $C(a,b) = C[a,b]-\{a,b\}$. 
		We use similar notions $P[a,b],P[a,b),P(a,b]$ and $P(a,b)$ for a path $P$.
		
		A digraph $D$ is \emph{connected} if its underlying graph is connected, and it is \emph{strongly connected} if $\dist_D(u,v)$ is finite for any two different vertices $u,v \in D$.

		We say a cycle $C$ is \emph{isometric} in $D$, if for every pair of $a,b \in C$, $\|C[a,b]\| = \dist_D(a,b)$ holds. Note that if $D$ is nonacyclic, then every cycle of length $g(D)$ is an isometric cycle. Let $\rho(C)$ be the number of vertices in a largest connected component of $D-C$. For integers $k,\ell \geq 1$, recall that $C(k,\ell)$ is a digraph formed by two internally disjoint paths $P_1$ and $P_2$ of lengths $k$ and $\ell$ such that $\init(P_1)=\init(P_2)$ and $\term(P_1)=\term(P_2)$.

		\section{Proof of Theorem~\ref{main theorem}}
		

		
		
		\noindent

		Actually, we will prove the following stronger statement as that in \cite{Picasarri-Arrieta}: for any digraph $D$ of $g(D)\geq 4k+2$, if $d_{D}^{+}(v_0)\geq 1$ and $d_{D}^{+}(v)\geq 2$ for each $v\in D-v_0$, then $D$ contains a subdivision of $C(k,k)$. Assume the statement is not true. Let $\mathcal D$ be the family of counterexamples, that is, each digraph in $\mathcal D$ meets the degree conditions and girth condition, but contains no subdivision of $C(k,k)$. Let $D\in \mathcal D$ be a smallest counterexample, i.e.,  $|D|$ is minimum and $\|D\|$ is also minimum if equality holds. We first show that  $d^{+}_{D}(v_0)=1$ and $d^{+}_{D}(v)=2$ for each $v\in D-v_0$. Otherwise, we can find a smaller counterexample by deleting an arc with tail $v_0$ if $d^{+}_{D}(v_0)\geq 2$ or an arc with tail $v$ if $d^{+}_{D}(v)\geq 3$, which contradicts the minimality of $D$.
		
		Next we shall demonstrate that $D$ is strongly connected. Suppose to the contrary that $D$ is not strongly connected. Hence, $D$ has $h(\geq 2)$ strongly connected components $D_{1},D_{2},\ldots, D_{h}$ such that there is no arc from $V(D_{j})$ to $V(D_{i})$ for any $1\leq i<j\leq h$. One can easily see that $D_{h}\in \mathcal D$, which is smaller than $D$, leading to a contradiction. We thus deduce that $D$ is strongly connected.

		
		
		
		
		\begin{claim}\label{cut}
			For $u$,$v\in D$, if $\dist_D(u,v)\leq 3k+2$, then there exists a $(v,u)$-cut-vertex.
		\end{claim}
		
		\begin{proof}[Proof of Claim~\ref{cut}.]
			If not, then by Menger's Theorem, there exist two internally disjoint $(v,u)$-paths $P_1$ and $P_2$. Since $g(D)\geq 4k+2$ and $\dist_D(u,v)\leq 3k+2$, it follows that both $P_1$ and $P_2$ have lengths at least $k$. Consequently, the union of $P_1$ and $P_2$ forms a subdivision of $C(k,k)$, which leads to a contradiction.
		\end{proof}
		
		Pick an isometric cycle $C$ among all isometric cycles in $D$ such that $\rho(C)$ is maximum. Then $C$ has length at least $g(D)\ge 4k+2$. Let $uv$ be an arbitrary arc of $C$, i.e., $\dist_D(u,v)=1<3k+2$. By Claim~\ref{cut}, $D$ contains a $(v,u)$-cut-vertex. Observe that all such $(v,u)$-cut-vertices must lie on $C$. Let $\ell(uv)$ be the number of $(v,u)$-cut-vertices, and let $c_1(uv),c_2(uv),\ldots,c_{\ell(uv)}(uv)$ be all the $(v,u)$-cut-vertices. For simplicity, we write $\ell$ for $\ell(uv)$ whenever there is no ambiguity. It is clear that $\ell\geq 1$ and any path with initial vertex $v$ and terminal vertex $u$ passes through all these vertices. We can assume without loss of generality that $\dist_{D}(v,c_{j}(uv))>\dist_{D}(v,c_{i}(uv))$ for every pair of $i,j$ with $1\leq i<j\leq \ell$. For simplicity, we define $c_0(uv)=v$ and $c_{\ell+1}(uv)=u$. Let $i_0$ be the least index such that $\dist_D(v,c_{i_0}(uv))\geq k$ and $i_1$ be the largest index such that $\dist_D(c_{i_1}(uv),u)\geq k$. Let $h=i_1-i_0$. It is convenient to denote $\tilde{c}_i(uv)=c_{i_0+i}(uv)$ for any $i\in [0,h]$.

		\begin{claim}\label{order}
			Every path from $v$ to $u$ passes through $c_{1}(uv),c_{2}(uv),\ldots,c_{\ell}(uv)$ in this order.
		\end{claim}
		\begin{proof}[Proof of Claim~\ref{order}.]
			On the contrary, there must be a path $P$ from $v$ to $u$ passing through two vertices $c_{i}(uv),c_{j}(uv)$ with $\|P[v,c_{j}(uv)]\|<\|P[v,c_{i}(uv)]\|$, where $1\leq i<j\leq \ell$. We can see that $P[v,c_{j}(uv)]\circ C(c_{j}(uv),u]$ is a path starting at $v$ and terminating at $u$ which avoids $c_{i}(uv)$. However, this contradicts that $c_{i}(uv)$ is a $(v,u)$-cut-vertex.
		\end{proof}
		
		\begin{claim}\label{no cut}
			There is no $(c_i(uv),c_{i+1}(uv))$-cut-vertex for every $i\in [0,\ell]$.
		\end{claim}
		
		\begin{proof}[Proof of Claim~\ref{no cut}.]
			Suppose not, we denote $c$ to be a $(c_i(uv),c_{i+1}(uv))$-cut-vertex for some $i\in [0,\ell]$. Arbitrarily choose a path $P_{0}$ from $v$ to $u$. By Claim~\ref{order}, we know that $P_{0}$ first passes through $c_i(uv)$ and then passes through $c_{i+1}(uv)$. Moreover, we have $c\in P_{0}[c_i(uv),c_{i+1}(uv)]$ since $c$ is a $(c_i(uv),c_{i+1}(uv))$-cut-vertex. Thereby, we deduce that $c$ lies on every path with initial vertex $v$ and terminal vertex $u$. It follows that $c$ is also a $(v,u)$-cut-vertex, which is a contradiction.
		\end{proof}
		
		\begin{claim}\label{dis}
			$\dist_D(c_i(uv),c_{i+1}(uv))\leq k-1$ for every $i\in [0,\ell]$.
		\end{claim}
		
		\begin{proof}[Proof of Claim~\ref{dis}.]	
			Assume the statement is not correct. Then there exists an integer $i\in [0,\ell]$ such that $\dist_D(c_i(uv),c_{i+1}(uv))\geq k$. By Claim~\ref{no cut} and Menger's Theorem, we deduce that there exist two internally disjoint $(c_i(uv),c_{i+1}(uv))$-paths $P_1$ and $P_2$. As $\dist_D(c_i(uv),c_{i+1}(uv))\geq k$, we deduce that $P_1$ and $P_2$ have length at least $k$. Then their union forms a subdivision of $C(k,k)$, leading to a contradiction.
		\end{proof}
		
		By the minimality of $i_0$, we have $\dist_D(v,c_{i_0-1}(uv))\leq k-1$. By Claim~\ref{dis}, $\dist_D(v,\tilde{c}_0(uv))=\dist_D(v,c_{i_0-1}(uv))+\dist_D(c_{i_0-1}(uv),\tilde{c}_0(uv))\leq 2k-2$. Similarly, $\dist_D(\tilde{c}_h(uv),u)\leq 2k-2$.
				
		\begin{claim}\label{wlog}
			There exists an arc $ab$ on $C$ such that $\dist_D(b,\tilde{c}_0(ab))=k$.
		\end{claim}
				
		\begin{proof}[Proof of Claim~\ref{wlog}.]
			Let $uv$ be an arbitrary arc of $C$. If $\dist_D(v,\tilde{c}_0(uv))=k$, then we are done by taking $ab=uv$. Suppose that $\dist_D(v,\tilde{c}_0(uv))\in (k,2k-2]$. Choose the arc $ab\in C$ such that $\dist_D(b,\tilde{c}_0(uv))=k$. We shall show that $\tilde{c}_0(uv)$ is also a $(b,a)$-cut-vertex, together with $\dist_D(b,\tilde{c}_0(uv))=k$, this implies that $\tilde{c}_0(uv)=\tilde{c}_0(ab)$.
			
			Suppose to the contrary that $\tilde{c}_0(uv)$ is not a $(b,a)$-cut-vertex. Let $A=C[b,\tilde{c}_0(uv))$ and $B=C[v,a]$.
			Since $D$ is strong, there exists an $(A,B)$-path in $D$. Moreover, as $\tilde{c}_0(uv)$ is not a $(b,a)$-cut-vertex, we can choose an $(A,B)$-path $P$ in $D-\tilde{c}_0(uv)$ whose internal vertices are disjoint from $A\cup B$.
			
			We first claim that  $V(P)\cap V\bigl(C-(A\cup B)\bigr)=\emptyset.$ Otherwise, let $c$ be a vertex in this intersection. Let $x=\init(P)\in A$ and $y=\term(P)\in B$. Then the concatenation $C[v,x]\circ P[x,c]\circ C[c,u]$ is a $(v,u)$-path avoiding $\tilde{c}_0(uv)$, contradicting that $\tilde{c}_0(uv)$ is a $(v,u)$-cut-vertex.
			
			Since $g(D)\geq 4k+2$ and $\|C[y,x]\|\leq \|C[v,\tilde{c}_0(uv)]\|=\dist_D(v,\tilde{c}_0(uv))\leq 2k-2$, it follows that $\|C[x,y]\|\geq 2k+4$. As $C$ is isometric, we have $\|P\|\geq \dist_D(x,y)=\|C[x,y]\|$. Consequently, the union of $P$ and $C[x,y]$ forms a subdivision of $C(k,k)$, a contradiction.
		\end{proof}

        From now on we fix the arc $ab$. To simplify notation, we omit $(ab)$ and write $\ell$, $c_i$, and $\tilde c_i$ for $\ell(ab)$, $c_i(ab)$, and $\tilde c_i(ab)$, respectively. See Figure 2 for an illustration.
		
		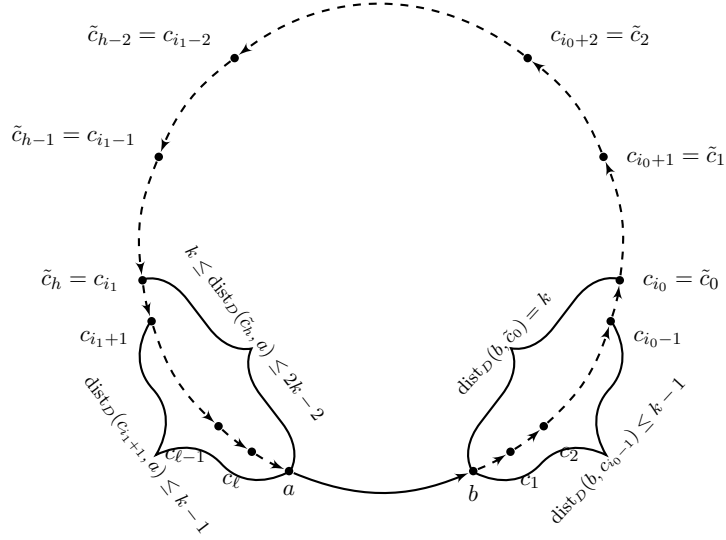
\begin{figure}[hbtp]
			\begin{center}	
				\begin{tikzpicture}[thick,scale=0.8, every node/.style={transform shape}]\label{fig_1}			
					\tikzset{vertex/.style = {circle,fill=black,minimum size=4pt, inner sep=0pt}}
					\tikzset{edge/.style = {->,> = latex'}}			
					\node[vertex, label=below:$a$] (v1) at (-112.5:4) {};
					\node[vertex, label=below:$b$] (v2) at (-67.5:4) {};
					\node[vertex, label={[label distance=0.4em]275:$c_1$}] (s1) at (-57.5:4) {};
					\node[vertex, label={[label distance=0.4em]300:$c_2$}] (s2) at (-47.5:4) {};
					\node[vertex, label={[label distance=0.4em]355:$c_{i_0-1}$}] (s3) at (-17.5:4) {};
					\node[vertex, label={[label distance=0.4em]0:$c_{i_0}=\tilde{c}_0$}] (s4) at (-7.5:4) {};
					\node[vertex, label={[label distance=0.4em]0:$c_{i_0+1}=\tilde{c}_1$}] (s5) at (22.5:4) {};
					\node[vertex, label={[label distance=0.4em]10:$c_{i_0+2}=\tilde{c}_2$}] (s6) at (52.5:4) {};
					\node[vertex, label={[label distance=0.4em]170:$\tilde{c}_{h-2}=c_{i_1-2}$}] (s7) at (127.5:4) {};
					\node[vertex, label={[label distance=0.4em]175:$\tilde{c}_{h-1}=c_{i_1-1}$}] (s8) at (157.5:4) {};
					\node[vertex, label={[label distance=0.4em]180:$\tilde{c}_h=c_{i_1}$}] (s9) at (-172.5:4) {};
					\node[vertex, label={[label distance=0.4em]185:$c_{i_1+1}$}] (s10) at (-162.5:4) {};
					\node[vertex, label={[label distance=0.4em]255:$c_{\ell-1}$}] (s11) at (-132.5:4) {};
					\node[vertex, label={[label distance=0.4em]260:$c_\ell$}] (s12) at (-122.5:4) {};			
					\draw[edge] (v1) to[out=-22.5, in=-157.5] (v2);
					\draw[edge,dashed] (v2) to[out=22.5, in=-147.5] (s1);
					\draw[edge,dashed] (s1) to[out=32.5, in=-137.5] (s2);
					\draw[edge,dashed] (s2) to[out=42.5, in=-107.5] (s3);
					\draw[edge,dashed] (s3) to[out=72.5, in=-97.5] (s4);
					\draw[edge,dashed] (s4) to[out=82.5, in=-67.5] (s5);
					\draw[edge,dashed] (s5) to[out=112.5, in=-37.5] (s6);
					\draw[edge,dashed] (s6) to[out=142.5, in=37.5] (s7);
					\draw[edge,dashed] (s7) to[out=-142.5, in=67.5] (s8);
					\draw[edge,dashed] (s8) to[out=-112.5, in=97.5] (s9);
					\draw[edge,dashed] (s9) to[out=-82.5, in=107.5] (s10);
					\draw[edge,dashed] (s10) to[out=-72.5, in=137.5] (s11);
					\draw[edge,dashed] (s11) to[out=-42.5, in=147.5] (s12);
					\draw[edge,dashed] (s12) to[out=-32.5, in=157.5] (v1);
					\draw[decorate,decoration={brace,amplitude=16pt,mirror,raise=1pt}] (3.9657794454952416445782301077143,-0.52210476888020636619362491158196)--(1.5307337294603590869138399361216,-3.6955181300451470245127327575872) node[midway,xshift=-20pt,yshift=15pt,rotate=47.5]{\footnotesize $\dist_D(b,\tilde{c}_0)=k$};
					\draw[decorate,decoration={brace,amplitude=32pt,mirror,raise=1pt}] (1.5307337294603590869138399361216,-3.6955181300451470245127327575872)--(3.814867802992907684575388258401,-1.2028231980170924864901885437243) node[midway,xshift=36pt,yshift=-26pt,rotate=47.5]{\footnotesize $\dist_D(b,c_{i_0-1})\leq k-1$};
					\draw[decorate,decoration={brace,amplitude=16pt,mirror,raise=1pt}] (-1.5307337294603590869138399361216,-3.6955181300451470245127327575872)--(-3.9657794454952416445782301077143,-0.52210476888020636619362491158196) node[midway,xshift=18pt,yshift=20pt,rotate=-52.5]{\footnotesize $k\leq\dist_D(\tilde{c}_h,a)\leq 2k-2$};
					\draw[decorate,decoration={brace,amplitude=32pt,mirror,raise=1pt}] (-3.814867802992907684575388258401,-1.2028231980170924864901885437243)--(-1.5307337294603590869138399361216,-3.6955181300451470245127327575872) node[midway,xshift=-35pt,yshift=-28pt,rotate=-52.5]{\footnotesize $\dist_D(c_{i_1+1},a)\leq k-1$};			
				\end{tikzpicture}
				\caption{The structure of $C$ in $D$. The solid arcs represent the arcs of $D$ and the dashed arcs represent the paths of $D$.}
			\end{center}
		\end{figure}
		

		
		Let $X_i=\{v\in V(D):\;v\;\text{can}\;\text{be}\;\text{reached}\;\text{from}\;\tilde{c}_i\;\text{via}\ \text{a}\;\text{path}\;\text{in}\;D-\tilde{c}_{i+1}\}$ for any $i\in [0,h]$. By definition, for each $i\in [0,h]$, we have $X_i\subseteq V(D-\tilde{c}_{i+1})$, and so $X_{i}\neq V(D)$.
		
		In the following context, we will consider the vertices $\tilde{c}_0,\tilde{c}_1,\ldots,\tilde{c}_h$ and the sets $X_0,X_1,\ldots,X_h$, and the subscripts will always modulo $h+1$.

    \begin{claim}\label{intersect_1}
			$X_i \cap V(C) = V\bigl(C[\tilde{c}_i,\tilde{c}_{i+1})\bigr)$ for each $i\in [0,h]$.
		\end{claim}
		\begin{proof}[Proof of Claim~\ref{intersect_1}.]	
			By definition of $X_i$, every vertex of $C[\tilde{c}_i,\tilde{c}_{i+1})$ belongs to $X_i$. Thus it suffices to prove that $X_i \cap V\bigl(C(\tilde{c}_{i+1},a]\cup C[b,\tilde{c}_i]\bigr)=\emptyset$.
			
			\vspace{0.1cm}
			\textbf{Case 1. $i\in[0,h-1]$}
			\vspace{0.1cm}
			
			Suppose to the contrary that there exists $v_1\in X_i \cap V\bigl(C(\tilde{c}_{i+1},a]\cup C[b,\tilde{c}_i]\bigr)$. By definition of $X_i$, there is a path $P_1$ from some $u_1\in C[\tilde{c}_i,\tilde{c}_{i+1})$ to $v_1$ internally disjoint from $C$.
			
			If $v_{1}\in C(\tilde{c}_{i+1},a]$, then the concatenation $C[b,u_{1}]\circ P[u_{1},v_{1}]\circ C[v_{1},a]$ is a $(b,a)$-path avoiding $\tilde{c}_{i+1}$, contradicting that $\tilde{c}_{i+1}$ is a $(b,a)$-cut-vertex.
						Otherwise, $v_{1}\in C[b,\tilde{c}_i)$, then $$\|P_{1}\|\geq \dist_{D}(u_{1},v_{1})=\|C[u_{1},v_{1}]\|\geq \|C[\tilde{c}_{h},b]\|=\dist_{D}(\tilde{c}_{h},b)\geq k+1. $$ Hence $P_{1} \cup C[u_{1},v_{1}]$ is a subdivision of $C(k,k)$, a contradiction.
			
			\vspace{0.1cm}
			\textbf{Case 2. $i=h$}
			\vspace{0.1cm}
			
			Similarly, suppose to the contrary that there exists $v_2\in X_h \cap V\bigl(C(\tilde{c}_0,\tilde{c}_h)\bigr)$. By definition of $X_h$, there exists a path $P_{2}$ from $u_{2}\in C[\tilde{c}_{h},\tilde{c}_{0})$ to $v_2$ internally disjoint from $C$.
			
			If $u_{2}\in C[b,\tilde{c}_{0})$, then the concatenation $C[b,u_{2}]\circ P[u_{2},v_{2}]\circ C[v_{2},a]$ is a $(b,a)$-path avoiding $\tilde{c}_0$, contradicting that $\tilde{c}_0$ is a $(b,a)$-cut-vertex.
						Otherwise, $u_{2}\in C[\tilde{c}_{h},a]$, then $$\|P_{2}\|\geq \dist_{D}(u_{2},v_{2})=\|C[u_{2},v_{2}]\|\geq \|C[a,\tilde{c}_{0}]\|=k+1. $$ It yields that $P_{2} \cup C[u_{2},v_{2}]$ forms a subdivision of $C(k,k)$, which also leads to a contradiction.
			
			Consequently, the claim follows.
		\end{proof}
		
		\begin{claim}\label{intersect_2}
			For any $i,j\in [0,h]$, the following hold:
			
			$(i).$ If $X_i\cap X_j\neq \emptyset$, then $j\in \{i-1,i+1\}$.
			
			$(ii).$ For $i\neq h-1$, if $X_i\cap X_{i+1}\neq \emptyset$, then $\tilde{c}_{i+1}\tilde{c}_{i+2}\in D$; if $X_{h-1}\cap X_{h}\neq \emptyset$, then $c_{i_1}c_{i_1+1}\in D$.
			
			$(iii).$ If $X_i\cap X_{i+1}\neq \emptyset$, then $v_0\in X_i\cup X_{i+1}$.
		\end{claim}
		
		\begin{proof}[Proof of Claim~\ref{intersect_2}.]
			$(i).$ Assume that $X_i\cap X_j\neq \emptyset$ for some $i,j\in [0,h]$. Observe that $X_i\cap X_j\neq V(D)$. Since $D$ is strongly connected, there must be an arc $uv\in D$ with $u\in X_i\cap X_j$ and $v\notin X_i\cap X_j$. We will divide the discussion into the following three cases.
			
			\vspace{0.1cm}
			\noindent\textbf{Case 1. $v\notin X_i\cup X_j$}
			\vspace{0.1cm}
			
			By the definition of $X_i$, there exists a $(\tilde{c}_{i},u)$-path $P$ in $D-\tilde{c}_{i+1}$. One can see that $P\circ uv$ is a $(\tilde{c}_{i},v)$-path in $D$. Since $v\notin X_i$, we must have $v=\tilde{c}_{i+1}$. Similarly, we can also deduce that $v=\tilde{c}_{j+1}$. But this is impossible.
			
			\vspace{0.1cm}
			\noindent\textbf{Case 2. $v\in X_i\setminus X_j$}
			\vspace{0.1cm}
			
			Due to $v\notin X_j$, we obtain $v=\tilde{c}_{j+1}$. By Claim~\ref{intersect_1}, we know that $\tilde{c}_{j+1}\in C[\tilde{c}_{i},\tilde{c}_{i+1})$ because $\tilde{c}_{j+1}\in X_i$. It follows that $j=i-1$.
			
			\vspace{0.1cm}
			\noindent\textbf{Case 3. $v\in X_j\setminus X_i$}
			\vspace{0.1cm}
			
			By analogous arguments, we obtain that $v=\tilde{c}_{i+1}$ and $\tilde{c}_{i+1}\in C[\tilde{c}_{j},\tilde{c}_{j+1})$, yielding that $j=i+1$. This proves $(i)$.

			Symmetrically, we can set $j=i+1$ and continue to use the notions in $(i)$.
			
			$(ii).$ Now we have $u\in X_i\cap X_{i+1}$ and $v=\tilde{c}_{i+1}$. By the definition of $X_{i+1}$, there exists a $(\tilde{c}_{i+1},u)$-path $P$. Let $x$ be the last vertex of $C[\tilde{c}_{i+1},\tilde{c}_{i+2})$ along $P$. Clearly, $P[x,u]$ is disjoint with $C$ except its ends. If $x\neq \tilde{c}_{i+1}$, then we get $\|C[x,\tilde{c}_{i+1}]\|\geq 3k+3$ since $\|C[\tilde{c}_{i+1},x]\|\leq \|C[\tilde{c}_{i+1},\tilde{c}_{i+2}]\|=\dist_D(\tilde{c}_{i+1},\tilde{c}_{i+2})\leq k-1$ and $g(D)\geq 4k+2$. Therefore, the union of $C[x,\tilde{c}_{i+1}]$ and $P[x,u]\circ u\tilde{c}_{i+1}$ (since $C$ is isometric, $\|P[x,u]\circ u\tilde{c}_{i+1}\| \geq \|C[x,\tilde{c}_{i+1}]\|$) forms a subdivision of $C(k,k)$, a contradiction. This implies that $x=\tilde{c}_{i+1}$ and it indicates that $P\circ u\tilde{c}_{i+1}$ is a cycle $C'$ satisfying that $V(C)\cap V(C')=\{\tilde{c}_{i+1}\}$\;(see Figure~\ref{fig_2} for an illustration).
			
			\begin{figure}[hbtp]
				\begin{center}	
					\begin{tikzpicture}[thick,scale=1, every node/.style={transform shape}]\label{fig_2}			
						\tikzset{vertex/.style = {circle,fill=black,minimum size=4pt, inner sep=0pt}}
						\tikzset{edge/.style = {->,> = latex'}}				
						\node[vertex, label=below:$a$] (v1) at (-112.5:4) {};
						\node[vertex, label=below:$b$] (v2) at (-67.5:4) {};
						\node[vertex, label={[label distance=0.4em]275:$c_1$}] (s1) at (-57.5:4) {};
						\node[vertex, label={[label distance=0.4em]300:$c_2$}] (s2) at (-47.5:4) {};
						\node[vertex, label={[label distance=0.4em]355:$c_{i_0-1}$}] (s3) at (-17.5:4) {};
						\node[vertex, label={[label distance=0.4em]0:$c_{i_0}=\tilde{c}_0$}] (s4) at (-7.5:4) {};
						\node[vertex, label={[label distance=0.4em]10:$\tilde{c}_i$}] (s5) at (52.5:4) {};
						\node[vertex, label={[label distance=0.4em]170:$\tilde{c}_{i+1}$}] (s6) at (117.5:4) {};			
						\node[vertex, label={[label distance=0.4em]170:$q$}] (s7) at (127.5:4) {};
						\node[vertex, label={[label distance=0.4em]170:$\tilde{c}_{i+2}$}] (s75) at (147.5:4) {};
						\node[vertex, label={[label distance=0.4em]175:$\tilde{c}_{h-1}=c_{i_1-1}$}] (s8) at (157.5:4) {};
						\node[vertex, label={[label distance=0.4em]180:$\tilde{c}_h=c_{i_1}$}] (s9) at (-172.5:4) {};
						\node[vertex, label={[label distance=0.4em]185:$c_{i_1+1}$}] (s10) at (-162.5:4) {};
						\node[vertex, label={[label distance=0.4em]255:$c_{\ell-1}$}] (s11) at (-132.5:4) {};
						\node[vertex, label={[label distance=0.4em]260:$c_\ell$}] (s12) at (-122.5:4) {};				
						\draw[edge] (v1) to[out=-22.5, in=-157.5] (v2);
						\draw[edge,dashed] (v2) to[out=22.5, in=-147.5] (s1);
						\draw[edge,dashed] (s1) to[out=32.5, in=-137.5] (s2);
						\draw[edge,dashed] (s2) to[out=42.5, in=-107.5] (s3);
						\draw[edge,dashed] (s3) to[out=72.5, in=-97.5] (s4);
						\draw[edge,dashed] (s4) to[out=82.5, in=-37.5] (s5);
						\draw[edge,dashed] (s5) to[out=142.5, in=27.5] (s6);
						\draw[edge] (s6) to[out=-152.5, in=37.5] (s7);
						\draw[edge,dashed] (s7) to[out=-142.5, in=57.5] (s75);
						\draw[edge,dashed] (s75) to[out=-122.5, in=67.5] (s8);
						\draw[edge,dashed] (s8) to[out=-112.5, in=97.5] (s9);
						\draw[edge] (s9) to[out=-82.5, in=107.5] (s10);
						\draw[edge,dashed] (s10) to[out=-72.5, in=137.5] (s11);
						\draw[edge,dashed] (s11) to[out=-42.5, in=147.5] (s12);
						\draw[edge,dashed] (s12) to[out=-32.5, in=157.5] (v1);				
						\draw[draw=red, thick] (-3.9657794454952416445782301077143,-0.52210476888020636619362491158196)--(-3.9657794454952416445782301077143,-4.5)--(3.9657794454952416445782301077143,-4.5)--(3.9657794454952416445782301077143,-0.7)--cycle;
						\node[red,font=\huge] at (-90:2){$X_h$};
						\draw[draw=red, thick] (3.5,-0.52210476888020636619362491158196)--(4.5,-0.52210476888020636619362491158196)--(4.5,0.52210476888020636619362491158196)--(3.5,0.52210476888020636619362491158196)--cycle;
						\node[red] at (0:4){$X_0$};
						\draw[draw=red, thick] (2.2497566339028876421292519346754,2.681155550916423123208374276944)--(2.8925442435894269684518953445827,3.4471999940354011584107669274994)--(3.6585886867084050036542879951381,2.8044123843488618320881235175921)--(3.0158010770218656773316445852308,2.0383679412298837968857308670367)--cycle;
						\node[red] at (43:4){$X_{i-1}$};
						\draw[draw=red, thick] (-4.1,1.5307337294603590869138399361216)--(-2,1.5307337294603590869138399361216)--(-2,-1.2)--(-4.1,-0.2)--cycle;
						\node[red] at (175:3){$X_{h-1}$};
						\node at (-2.4,-0.8){$v_0$};				
						\begin{scope}[shift={(117.5:3)}]
							\node[vertex, label={[label distance=0.1em]200:$p$}] (p1) at (180:1) {};
							\node[vertex, label={[label distance=0.1em]200:$w$}] (p15) at (230:1) {};					
							\node[vertex, label={[label distance=0.1em]0:$u$}] (p2) at (50:1) {};
							\draw[edge] (s6) to[out=-152.5, in=90] (p1);
							\draw[edge] (p2) to[out=140, in=15.5] (s6);
							\draw[thick,edge,dashed] (180:1) arc[start angle=180, end angle=230, radius=1];
							\draw[thick,edge,dashed] (230:1) arc[start angle=230, end angle=410, radius=1];
						\end{scope}				
						\node[font=\huge] at (117.5:3){$C'$};			
					\end{tikzpicture}
					\caption{The structure of $C$ in $D$. The solid arcs represent the arcs of $D$ and the dashed arcs represent the directed paths of $D$. We demonstrate the situation $X_{h-1}\cap X_h\neq \emptyset$ in this figure.}
				\end{center}
			\end{figure}
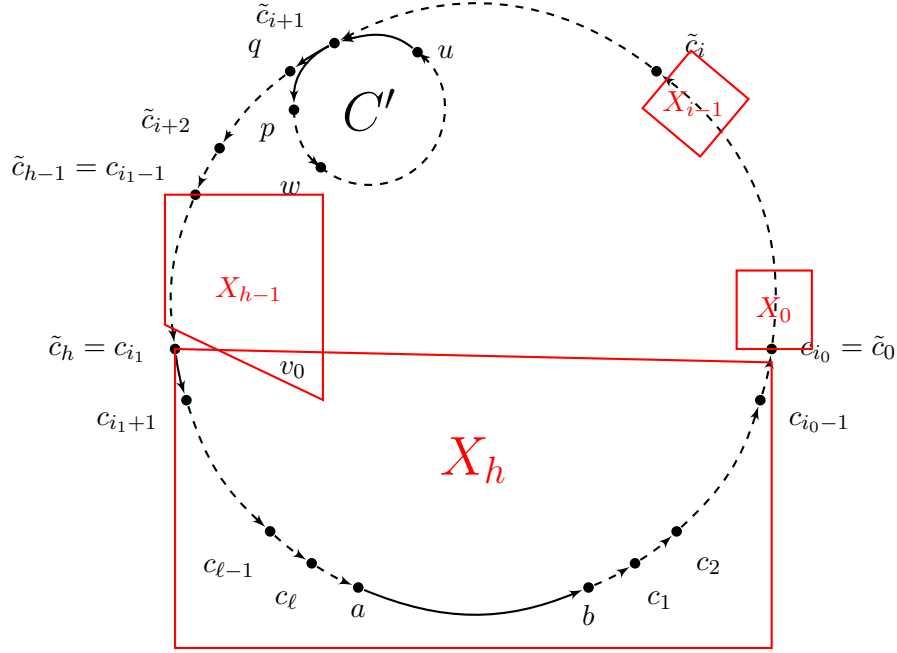
			
			Let $w$ be the first vertex of $X_i$ along $C'$ from $\tilde{c}_{i+1}$ to $u$. Then $w\in C'(\tilde{c}_{i+1},u]$. Let $Q$ be a shortest $(\tilde{c}_{i},w)$-path in $D-\tilde{c}_{i+1}$. By the definition of $X_i$, we have $V(Q)\subseteq X_i$, and by Claim~\ref{intersect_1}, 
            we deduce that $Q$ is disjoint with $C[\tilde{c}_{i+1},\tilde{c}_{i})$. Moreover, $Q$ intersects $C'[\tilde{c}_{i+1},w]$ only at $w$ since otherwise this will contradicts the choice of $w$. If $\|C'[\tilde{c}_{i+1},w]\|\geq k$, then the union of $C[\tilde{c}_{i+1},\tilde{c}_{i}]\circ Q$ and $C'[\tilde{c}_{i+1},w]$ forms a subdivision of $C(k,k)$, a contradiction. Thereby, $\|C'[\tilde{c}_{i+1},w]\|\leq k-1$ and $\|C'[w,\tilde{c}_{i+1}]\|\geq 3k+3$ since $g(D)\ge 4k+2$.
			
			We shall illustrate that there exists no path in $D-\tilde{c}_{i+1}$ with initial vertex in $V(C')$ and terminal vertex in $V(C)$. Suppose, on the contrary, that the statement is false. Arbitrarily select a path $R$ internally disjoint from $C$ and $C'$ and denote $y=\init(R)$ and $z=\term(R)$. Since $C'[\tilde{c}_{i+1},y]\circ R[y,z]$ is a path from $\tilde{c}_{i+1}$ to $z$ which shares no common vertex of $C$ except its ends, Claim~\ref{intersect_1} implies that $z\in C(\tilde{c}_{i+1},\tilde{c}_{i+2}]$. In addition, we will show that $y\in C'[\tilde{c}_{i+1},w)$. If not, then $Q[\tilde{c}_i,w]\circ C'[w,y]\circ R[y,z]$ is a $(\tilde{c}_i,z)$-path avoiding $\tilde{c}_{i+1}$. It follows that $z\in X_i$, which contradicts Claim~\ref{intersect_1}. This proves $y\in C'[\tilde{c}_{i+1},w)$. Since $\|C'[w,\tilde{c}_{i+1}]\|\geq 3k+3$, the union of $C'[y,u]\circ u\tilde{c}_{i+1}$ and $R[y,z]\circ C[z,\tilde{c}_{i+1}]$ forms a subdivision of $C(k,k)$, a contradiction.
			
			For $i\neq h-1$, if $\tilde{c}_{i+1}\tilde{c}_{i+2}\notin D$, then there exists a vertex $q\in C(\tilde{c}_{i+1},\tilde{c}_{i+2})$ such that $\tilde{c}_{i+1}q\in C$. Since $d^+_{D}(\tilde{c}_{i+1})=2$, there is exactly one other out-neighbor of $\tilde{c}_{i+1}$, denote  by $p$, and note that $p\in C'$. Take all $\tilde{c}_{i+1}$ to $\tilde{c}_{i+2}$ paths in $D-q$. Evidently, every such $\tilde{c}_{i+1}$ to $\tilde{c}_{i+2}$ path must pass through $p$. Consequently, each of them contains a segment from $C'$ to $C$ in $D-\tilde{c}_{i+1}$, contradicting that there exists no path in $D-\tilde{c}_{i+1}$ with initial vertex in $V(C')$ and terminal vertex in $V(C)$. Therefore, there is no $(\tilde{c}_{i+1},\tilde{c}_{i+2})$-path in $D-q$. This yields that $q$ is a $(\tilde{c}_{i+1},\tilde{c}_{i+2})$-cut-vertex, which contradicts Claim~\ref{no cut}.
			
			We next consider $i=h-1$ and suppose to the contrary that $c_{i_1}c_{i_1+1}\notin D$. Similarly, let $q\in C(c_{i_1},c_{i_1+1})$ be a vertex such that $c_{i_1}q\in C$ and $p\in C'$ be another out-neighbor of $c_{i_1}$. If there exists a $(c_{i_1},c_{i_1+1})$-path in $D-q$, then such path must pass through $p$, yielding a  path with initial vertex in $V(C')$ and terminal vertex in $V(C)$ that avoids $c_{i_1}$, which is a contradiction. Thus $q$ is a $(c_{i_1},c_{i_1+1})$-cut-vertex. But it contradicts Claim~\ref{no cut}, this proves $(ii)$.
			
			
			(iii). We retain the notation of $p$ and $q$ from (ii). Let $$U=\{v\in D: \text{there exists a $(\tilde{c}_{i},v)$-path  in $D-\tilde{c}_{i+1}q$}\}.$$ 
			For any $s\in U$, by definition, there exists a $(\tilde{c}_{i},s)$-path $P'$ that does not pass through $\tilde{c}_{i+1}q$. If $\tilde{c}_{i+1}\notin P'$, then $s\in X_i$. If $\tilde{c}_{i+1} \in P'$ and $\tilde{c}_{i+2}\notin P'$, then $s\in X_{i+1}$. If both $\tilde{c}_{i+1}$ and $\tilde{c}_{i+2}$ belong to $P'$, then $P'$ must first pass through $\tilde{c}_{i+1}$ and then $\tilde{c}_{i+2}$ because, otherwise, $\tilde{c}_{i+2}\in X_i$, contradicting Claim~\ref{intersect_1}. As $P'$ does not pass through $\tilde{c}_{i+1}q$, it must pass through $p$. This indicates that $P'$ contains a subpath in $D-\tilde{c}_{i+1}$ with initial vertex in $V(C')$ and terminal vertex in $V(C)$, which leads to a contradiction. Consequently, $U\subseteq X_i\cup X_{i+1}$.
			
			We suppose that $v_0\notin X_i\cup X_{i+1}$. It is easily seen that $g(D[U])\geq g(D)$, and $d^+_{D[U]}(\tilde{c}_{i+1})=1$ and every other vertex has out-degree two. This means that $D[U]$ is a smaller counterexample, which is a contradiction. Therefore, this proves $(iii)$.
		\end{proof}
		
		Now we would like to estimate the upper bound of $\|C\|$ as below:
		\begin{align*}
			\|C\|&= \dist_D(b,\tilde{c}_{0})+\sum\limits_{i=0}^{h-1} \dist_D(\tilde{c}_{i},\tilde{c}_{i+1})+\dist_D(\tilde{c}_{h},a)+1\\
			&\leq k+h(k-1)+(2k-2)+1\\
			&=(h+3)k-(h+1).
		\end{align*}
		
		\noindent{\bf Remark (A)}: By Claim~\ref{intersect_2} $(ii)$, we know that $\tilde{c}_{i+1}\tilde{c}_{i+2}\in D$ for each $i\neq h-1$ with $X_i\cap X_{i+1}\neq \emptyset$, thus reducing the upper bound estimation of $\dist_D(\tilde{c}_{i+1},\tilde{c}_{i+2})$ from at most $k-1$ to $1$. Similarly, for $X_{h-1}\cap X_h\neq \emptyset$, we have $\tilde{c_{h}}c_{i_1+1}\in D$, which implies $\dist_D(\tilde{c}_{h},a)=k$ as $\dist_D(c_{i_1+1},a)\leq k-1$ and $\dist_D(\tilde{c_{h}},a)\geq k$. That is, the upper bound of $\|C\|$ can be further improved .
		
		Let $A$ be a largest connected component of $D-C$, i.e., $|A|=\rho(C)$.
		\begin{claim}\label{main}
			There exists a $\lambda\in [0,h]$ satisfying that
			
			$(i).$ $v_0\notin X_{\lambda}$;
			
			$(ii).$ $A\cap X_{\lambda}=\emptyset$;
			
			$(iii).$ $X_{\mu}\cap X_{\lambda}=\emptyset$ for all $\mu\neq\lambda$.
		\end{claim}
		
		\begin{proof}[Proof of Claim~\ref{main}.]
			Since $g(D)\geq 4k+2$, we obtain that $h\geq 2$. We first show the following statement:
			\begin{equation}\label{stat_2}		
				\text{If $X_i\cap X_j=\emptyset$, then there is no arc between $X_i\setminus C$ and $X_j\setminus C$.}	
			\end{equation}
						Suppose the statement is not correct. We may assume without loss of generality that there exists an arc $v_iv_j$ with $v_i \in X_i\setminus C$ and $v_j \in X_j\setminus C$. By the definition of $X_i$, there exists a $(\tilde{c}_{i},v_i)$-path in $D-\tilde{c}_{i+1}$. Together with $v_iv_j$ we find a $(\tilde{c}_{i},v_j)$-path in $D-\tilde{c}_{i+1}$. It follows that $v_j\in X_i$, which contradicts that $X_i\cap X_j=\emptyset$.
			
			Here to proceed with the proof, we will consider the following four cases.
			
			\vspace{0.1cm}
			\noindent\textbf{Case 1. $h=2$}
			\vspace{0.1cm}
			
			The sets $X_0, X_1, X_2$ must be pairwise disjoint. Otherwise, from Remark (A)
            we can derive that $\|C\|\leq (h+3)k-(h+1)-(k-2)=(5k-3)-(k-2)=4k-1<g(D)$, leading to a contradiction. By (\ref{stat_2}), there is no arc between every pair of $X_0\setminus C$, $X_1\setminus C$ and $X_2\setminus C$. Let $A \subseteq X_{\alpha} \setminus C$ and $v_0 \in X_{\beta}$, where $\alpha,\beta\in [0,2]$. We can take another integer different from $\alpha$ and $\beta$ in $[0, 2]$ to be $\lambda$.
			
			\vspace{0.1cm}
			\noindent\textbf{Case 2. $h=3$}
			\vspace{0.1cm}
			
			If $X_0, X_1,X_2, X_3$ are pairwise disjoint, then by using analogous discussion as \textbf{Case 1} we can find a desired $\lambda$. Next, we take the case that $X_0, X_1,X_2, X_3$ are not pairwise disjoint. We assert that at most one pair of $i,j\in[0,3]$ satisfies that $X_i\cap X_j\neq \emptyset$. Suppose not, there exist at least two such pairs. By Remark (A), we have $\|C\|\leq (h+3)k-(h+1)-2(k-2)=(6k-4)-2(k-2)=4k<g(D)$, leading to a contradiction. This proves the statement and we can assume that $X_\alpha\cap X_{\alpha+1}\neq \emptyset$ for some $\alpha$.
			
			Observe that $X_\alpha\cup X_{\alpha+1}$, $X_{\alpha+2}$ and $X_{\alpha+3}$ are pairwise disjoint. By (\ref{stat_2}), there exists no arc between any pair of $(X_\alpha\cup X_{\alpha+1})\setminus C$, $X_{\alpha+2}\setminus C$ and $X_{\alpha+3}\setminus C$. Thus, $A$ must lie within one of these three sets. 
			
			By Claim~\ref{intersect_2} $(iii)$, one sees that $v_0\in X_\alpha\cup X_{\alpha+1}$. Moreover, for $\beta\in\{\alpha+2,\alpha+3\}$, we have $X_\beta\cap (X_\alpha\cup X_{\alpha+1})=(X_\beta\cap X_\alpha)\cup (X_\beta\cap X_{\alpha+1})=\emptyset$, which yields that $v_0\notin X_\beta$. Now we can select $\lambda\in \{\alpha+2,\alpha+3\}$ with $A\cap X_{\lambda}=\emptyset$. 
			
			\vspace{0.1cm}
			\noindent\textbf{Case 3. $h=4$}
			\vspace{0.1cm}
			
			If $X_0, X_1,\ldots, X_4$ are pairwise disjoint or there exists exactly one pair of $i,j\in [0,4]$ such that $X_i\cap X_j\neq \emptyset$, then we can find a desired $\lambda$ similar to the previous arguments. We claim that there are exactly two pairs of $i,j\in[0,4]$ satisfying that $X_i\cap X_j\neq \emptyset$. Suppose to the contrary that there exist at least three such pairs, one can deduce that $|C|\leq (h+3)k-(h+1)-3(k-2)=(7k-5)-3(k-2)=4k+1<g(D)$, which is a contradiction. Let $\alpha,\beta \in [0,4]$ be integers with $X_\alpha\cap X_{\alpha+1}\neq \emptyset$ and $X_\beta\cap X_{\beta+1}\neq \emptyset$.
			
			By Claim~\ref{intersect_2} $(iii)$, we get $v_0\in X_\alpha\cup X_{\alpha+1}$ and $v_0\in X_\beta\cup X_{\beta+1}$. If $\beta\notin \{\alpha-1,\alpha+1\}$, then $v_0\in (X_\alpha\cup X_{\alpha+1})\cap (X_\beta\cup X_{\beta+1})=(X_\alpha\cap X_\beta)\cup (X_\alpha\cap X_{\beta+1})\cup (X_{\alpha+1}\cap X_\beta)\cup (X_{\alpha+1}\cap X_{\beta+1})=\emptyset$, which is impossible. By symmetry, we may assume $\beta=\alpha+1$. Additionally, for $\gamma\in\{\alpha+3,\alpha+4\}$, we have $X_\gamma\cap (X_\alpha\cup X_{\alpha+1}\cup X_{\alpha+2})=(X_\gamma\cap X_\alpha)\cup (X_\gamma\cap X_{\alpha+1})\cup (X_\gamma\cap X_{\alpha+2})=\emptyset$, which implies $v_0\notin X_\gamma$.
			
			Observe that $X_\alpha\cup X_{\alpha+1}\cup X_{\alpha+2}$, $X_{\alpha+3}$ and $X_{\alpha+4}$ are pairwise disjoint, hence by (\ref{stat_2}), there is no arc among $(X_\alpha\cup X_{\alpha+1}\cup X_{\alpha+2})\setminus C$, $X_{\alpha+3}\setminus C$ and $X_{\alpha+4}\setminus C$. This indicates that $A$ must lie within one of them, and we can select $\lambda\in \{\alpha+3,\alpha+4\}$ such that $A\cap X_{\lambda}=\emptyset$.
			
			\vspace{0.1cm}
			\noindent\textbf{Case 4. $h\geq 5$}
			\vspace{0.1cm}
			
			If there are at most two pairs of $i,j\in [0,h]$ satisfying that $X_i\cap X_j\neq \emptyset$, then one can also get a desired $\lambda$ analogous to the previous arguments. We next suppose that there are distinct $\alpha,\beta$ and $\gamma$ with  $X_\alpha\cap X_{\alpha+1}\neq \emptyset$, $X_\beta\cap X_{\beta+1}\neq \emptyset$ and $X_\gamma\cap X_{\gamma+1}\neq \emptyset$. Similar to that of \textbf{Case 3}, we have $\beta\in\{\alpha-1,\alpha+1\}$ and $\gamma\in\{\beta-1,\beta+1\}$. We may assume without loss of generality that $\gamma=\alpha+2$ and $\beta=\alpha+1$. By Claim~\ref{intersect_2} $(iii)$, we obtain that  $v_0\in X_\alpha\cup X_{\alpha+1}$, $v_0\in X_{\alpha+1}\cup X_{\alpha+2}$ and $v_0\in X_{\alpha+2}\cup X_{\alpha+3}$. In addition, by Claim~\ref{intersect_2} $(i)$, we have $X_\alpha\cap X_{\alpha+2}=\emptyset$, $X_{\alpha+1}\cap X_{\alpha+3}=\emptyset$ and $X_\alpha\cap X_{\alpha+3}=\emptyset$. Thus we conclude that $v_0\in X_{\alpha+1}\cap X_{\alpha+2}$ and $v_0\notin X_\alpha,X_{\alpha+3}$.
			
			If there exists another $i\notin \{\alpha,\alpha+1,\alpha+2\}$ such that $X_i\cap X_{i+1} \neq \emptyset$, then $v_0\in X_i\cup X_{i+1}$ by Claim~\ref{intersect_2} $(iii)$. We thus obtain that $v_0\in (X_i\cup X_{i+1})\cap X_{\alpha+1}\cap X_{\alpha+2}=(X_i\cap X_{\alpha+1}\cap X_{\alpha+2})\cup (X_{i+1}\cap X_{\alpha+1}\cap X_{\alpha+2})=\emptyset$. But it is impossible.
			
			On the other hand, for any $j\in \{\alpha+4,\alpha+5,\ldots,\alpha+h\}$, one can deduce that $X_j\cap (X_{\alpha}\cup X_{\alpha+1}\cup X_{\alpha+2}\cup X_{\alpha+3})=(X_j\cap X_\alpha)\cup (X_j\cap X_{\alpha+1})\cup (X_j\cap X_{\alpha+2})\cup (X_j\cap X_{\alpha+3})=\emptyset$. It follows that $v_0\notin X_j$.
			
			Observe that $X_{\alpha}\cup X_{\alpha+1}\cup X_{\alpha+2}\cup X_{\alpha+3}, X_{\alpha+4},\ldots,X_{\alpha+h}$ are pairwise disjoint. By (\ref{stat_2}), there is no arc between every pair of $(X_\alpha\cup X_{\alpha+1}\cup X_{\alpha+2}\cup X_{\alpha+3})\setminus C$, $X_{\alpha+4}\setminus C,\ldots, X_{\alpha+h}\setminus C$. As a consequence, $A$ must lie within one of them. We can pick $\lambda\in \{\alpha+4,\alpha+5,\ldots,\alpha+h\}$ with $A\cap X_{\lambda}=\emptyset$.
		\end{proof}

		Choose such a $\lambda\in[0,h]$ guaranteed by Claim~\ref{main} and let $s\in C[\tilde{c}_{\lambda},\tilde{c}_{\lambda+1})$ with $s\tilde{c}_{\lambda+1}\in D$. Define $S=\{v\in D:\;\text{there}\;\text{exists}\;\text{an}\;(s,v)\text{-path}\;\text{in}\;D-\tilde{c}_{\lambda+1}\}$. Clearly, we have $S\subseteq X_{\lambda}$. Since $v_0\notin X_{\lambda}$, $D[S]$ satisfies that $\delta^+(D[S])\geq 1$, implying that $D[S]$ contains a cycle. Let $C'$ be an isometric cycle in $D[S]$. We shall illustrate that $C'$ is also an isometric cycle in $D$. If not, then there exist $x,y\in C'$, and an $(x,y)$-path $P$ which shares no common vertex with $C'$ such that $\|P\|<\|C'[x,y]\|$. Due to $C'$ is isometric in $D[S]$, $P$ must contain some vertex outside $S$. Let $v$ be the first vertex of $P$ outside $S$. Note that $v$ can be reached from $s$ via a path. As $v\notin S$, we deduce that $v=\tilde{c}_{\lambda+1}$. Let $\tilde{c}_{\mu}$ be the last $(b,a)$-cut-vertex along $P$ from $x$ to $y$. Then $P[\tilde{c}_{\mu},y]$ is a $(\tilde{c}_{\mu},y)$-path which avoids $\tilde{c}_{\mu+1}$. Hence, $y\in X_{\mu}$. As $y\in S\subseteq X_{\lambda}$, we get $y\in X_{\mu}\cap X_{\lambda}$. By Claim~\ref{main} (iii), we obtain that $\mu=\lambda$. Therefore, $k\leq \dist(\tilde{c}_{\lambda+1},\tilde{c}_{\lambda})\leq \|P[\tilde{c}_{\lambda+1},\tilde{c}_{\lambda}]\|\leq \|P\|<\|C'[x,y]\|$. It yields that the union of $P$ and $C'[x,y]$ forms a subdivision of $C(k,k)$, which leads to a contradiction. This shows that $C'$ is an isometric cycle in $D$.
		
		Recall that $D$ is strongly connected and $A$ is a component of $D-C$. There must be an arc $cv\in D$ with $c\in C$ and $v\in A$. If $c\in X_{\lambda}$, then $c\in C[\tilde{c}_{\lambda},\tilde{c}_{\lambda+1})$ by Claim~\ref{intersect_1}. Moreover, $v$ can be reached from $\tilde{c}_\lambda$ avoiding $\tilde{c}_{\lambda+1}$, implying that $v \in X_\lambda$. However, by Claim~\ref{main}, we deduce that $A \cap X_\lambda = \emptyset$, which is a contradiction. Consequently, we have $c\notin X_{\lambda}$, and so $c\notin C'$. This indicates that $D-C'$ has a larger component containing $A\cup \{c\}$, contradicting the maximality of $A$.

		The proof of Theorem~\ref{main theorem} is now completed.

		
		\noindent

		\section{Remarks and discussions}
		
		In this paper, we mainly prove that every digraph of minimum out-degree at least two and girth at least $4k+2$ contains a subdivision of $C(k,k)$ for each positive ingeter $k$, and a family of digraphs of minimum out-degree two and girth $k$ in which every digraph contains no subdivision of $C(k,k)$. These improve several results of Picasarri-Arrieta and Rambaud \cite{Picasarri-Arrieta}. Furthermore, Problem~\ref{girth} can be updated to the following:
		\begin{problem}\label{new}
			Find the minimum value $g\in [k+1,4k+2]$ such that mader$_{\delta^+}^{(g)}(C(k,k))\leq 2$. 
		\end{problem}

		\subsection*{Acknowledgements}
		This work was supported by the National Key Research and Development Program of China (2023YFA1010203), the National Natural Science Foundation of China (12471336, 12501473, 12401455, 12571369), and the Innovation Program for Quantum Science and Technology (2021ZD0302902).
		

		\vskip 3mm
	\end{spacing}
	
\end{document}